\documentclass[a4paper,12pt]{amsart}

\usepackage{amsthm,amssymb,latexsym,mathrsfs}
\usepackage{enumerate}
\input amssym.def

\newtheorem{theorem}{T{\hskip 0pt\footnotesize\bf HEOREM}}[section]
\newtheorem{lemma}[theorem]{L{\hskip 0pt\footnotesize\bf EMMA}}
\newtheorem{proposition}[theorem]{P{\hskip 0pt\footnotesize\bf ROPOSITION}}
\newtheorem{definition}[theorem]{D{\hskip 0pt\footnotesize\bf EFINITION}}
\newtheorem{corollary}[theorem]{C{\hskip 0pt\footnotesize\bf OROLLARY}}
\newtheorem{remark}[theorem]{R{\hskip 0pt\footnotesize\bf EMARK}}

\numberwithin{equation}{section}

\scrollmode

\begin{document}

\title[Weak factorization for Bergman-Orlicz spaces]{Atomic decomposition and Weak Factorization  for Bergman-Orlicz spaces}

\author[D. Bekolle]{David Bekolle}
\address{Department of Mathematics, Faculty of Science, University of Ngaoundere, P. O. Box 484 Ngaoundere, Cameroon. }
\email{dbekolle@gmail.com}
\author[A. Bonami]{Aline Bonami}
\address{Federation Denis Poisson, MAPMO CNRS-UMR 7349, Universit\'e d'Orl\'eans, 45067 Orl\'eans Cedex 2, France}
\email{aline.bonami@gmail.com}
\author[E. Tchoundja]{Edgar Tchoundja}
\address{Department of Mathematics, Faculty of Science, University of Yaound\'e I, P. O. Box 812  Yaound\'e, Cameroon.}
\email{tchoundjaedgar@yahoo.fr}
\curraddr{Edgar Tchoundja, Department of Mathematics, Washington University in St. Louis, One Brookings Drive, St. Louis, MO. 63130. }
\email{etchoundja@math.wustl.edu}

\begin{abstract} For $\mathbb B^n$ the unit ball of $\mathbb C^n$, we consider Bergman-Orlicz spaces
of holomorphic functions in $L^\Phi_\alpha(\mathbb B^n)$, which are generalizations of classical Bergman spaces.
We obtain atomic decomposition for functions in the Bergman-Orlicz space $\mathcal A^\Phi_\alpha (\mathbb B^n)$ where
$\Phi$ is either convex or concave growth function. We then prove weak factorization theorems involving the Bloch space and a Bergman-Orlicz space and also weak factorization theorems involving two Bergman-Orlicz spaces.
\end{abstract}


\maketitle

\section{Introduction and main results}

Let  $\mathbb B^n$ be the unit ball  of $\mathbb C^n$. We denote by $d\nu$ the Lebesgue measure on
$\mathbb B^n$. The space $\mathcal H(\mathbb B^n)$ is the
set of holomorphic functions on $\mathbb B^n.$

For $z=(z_1,\cdots,z_n)$ and $w=(w_1,\cdots,w_n)$ in
$\mathbb C^n$, we let
$$\langle z,w\rangle =z_1\overline {w_1} +
\cdots + z_n\overline {w_n}$$
so that $|z|^2=\langle
z,z\rangle =|z_1|^2 +\cdots +|z_n|^2$.

  We say that a function $\Phi$ is a growth function if it is a continuous and non-decreasing function
  from $[0,\infty)$ onto itself.

For $\alpha>-1$, we denote by $d\nu_{\alpha}$ the normalized Lebesgue measure
$d\nu_{\alpha}(z)=c_{\alpha}(1-|z|^2)^{\alpha}d\nu(z)$, with $c_\alpha$
such
that $\nu_\alpha(\mathbb B^n)=1$. For $\Phi$ a growth function, the weighted Orlicz
space $L_\alpha^{\Phi}(\mathbb B^n)$ is  the space of measurable functions $f$ such that, there exists a $\lambda>0$ such that
 $$\int_{\mathbb B^n}\Phi\left(\frac{|f(z)|}\lambda\right)d\nu_{\alpha}(z)<\infty.$$
We define on $L_\alpha^{\Phi}(\mathbb B^n)$ the following Luxembourg (quasi)-norm
\begin{equation}\label{BergOrdef1}
\|f\|_{\Phi,\alpha}:=\inf\{\lambda>0: \int_{\mathbb B^n}\Phi\left(\frac{|f(z)|}{\lambda}\right)d\nu_{\alpha}(z)\le 1\}
\end{equation}
which is finite for $f\in L_\alpha^{\Phi}(\mathbb B^n)$ (see \cite{sehbastevic}).  The weighted Bergman-Orlicz space $\mathcal A_\alpha^{\Phi}(\mathbb B^n)$ is the subspace
of $L_\alpha^{\Phi}(\mathbb B^n)$ consisting of holomorphic functions.

When  $\Phi(t)=t^p$, we recover the classical weighted Bergman spaces denoted by
$\mathcal A_\alpha^{p}(\mathbb B^n)$   and defined by
$$\|f\|_{p,\alpha}^p:= \int_{\mathbb B^n}|f(z)|^pd\nu_{\alpha}(z)<\infty.$$

We say that a growth function $\Phi$ is of lower type $p > 0$ if there exists $C>0$ such that, for $s>0$ and $0<t\le 1$,
\begin{equation}\label{eq:lowertype}
 \Phi(st)\le Ct^p\Phi(s).\end{equation}

We denote by $\mathscr{L}$ the set of growth functions $\Phi$ of lower type $p$ for some $p, \hskip 1truemm 0<p\leq 1,$ 
such that the
function $t\mapsto \frac{\Phi(t)}{t}$ is non-increasing. We also denote by $\mathscr{L}_p, 0<p\leq 1,$ the subset of $\mathscr{L}$ 
consisting of growth functions of lower type $p.$ 

We say that a growth function $\Phi$ is of upper type  $q>0$ if there exists $C>0$ such that, for $s>0$ and $t\ge 1$,
\begin{equation}\label{uppertype}
 \Phi(st)\le Ct^q\Phi(s).\end{equation}
We denote by $\mathscr{U}$ the set of growth functions $\Phi$ of upper type $q$ for some $q, \hskip 1truemm q\geq 1,$
such that the function $t\mapsto \frac{\Phi(t)}{t}$ is non-decreasing. We also denote by $\mathscr{U}^q, q\geq 1,$ the subset of $\mathscr{U}$ 
consisting of growth functions of upper type $q.$

We say that $\Phi$ satisfies the $\Delta_2-$condition if there exists a constant $K>1$ such that, for any $t\ge 0$,
\begin{equation}\label{eq:delta2condition}
 \Phi(2t)\le K\Phi(t).\end{equation}
It is easy to see the equivalence between (\ref{uppertype}) and (\ref{eq:delta2condition}). Moreover, if the function $t\mapsto \frac{\Phi(t)}{t}$ is non-increasing, then $\Phi$ satisfies the $\Delta_2-$condition; this is the case when $\Phi\in \mathscr L.$\\
Recall that two growth functions $\Phi_1$ and $\Phi_2$ are said to be equivalent if there exists
some constant $c$ such that
$$c\Phi_1(ct) \le \Phi_2(t)\le c^{-1}\Phi_1(c^{-1}t).$$
Such equivalent growth functions define the same Orlicz space.
Note that we may always suppose that any $\Phi\in \mathscr{L}_p$ (resp. $\mathscr{U}^q$)  is concave (resp. convex) and
that $\Phi$ is a $\mathscr{C}^1$ function with derivative $\Phi^{\prime}(t)\simeq \frac{\Phi(t)}{t}$ (see \cite{BS} for the lower type functions).

\begin{remark}\label{rmk}
Given a growth function $\Phi$, we recall that the upper and lower indices,  $a_\Phi$ and $b_\Phi$ respectively, of $\Phi$ are defined by:
$$a_\Phi=\sup\left\lbrace p: \inf_{t\geq 1,\,\lambda>0}\frac{\Phi(\lambda t)}{t^p\Phi(\lambda)}>0\right\rbrace, \hskip 2truemm b_\Phi=\inf\left\lbrace q: \sup_{t\geq 1,\,\lambda>0}\frac{\Phi(\lambda t)}{t^q\Phi(\lambda)}<\infty\right\rbrace $$
We say that $\Phi$ is of finite lower  (resp. upper) type if $a_\Phi<\infty$ (resp. $b_\Phi<\infty$). In this case, $\Phi$ is of lower type $p$ (resp. of upper type $q)$ for every $p<a_\Phi$  (resp. for every $q>b_\Phi).$
\end{remark}

\vskip .2cm

Our first interest in this paper is to obtain atomic decomposition theorems for functions in
$\mathcal A_{\alpha}^{\Phi}(\mathbb B^n)$. For $p>0$, atomic decomposition for functions in
$\mathcal A_{\alpha}^{p}(\mathbb B^n)$ is a well known result, see \cite[Theorem 2.30]{KZ}. Our first main result extends the atomic decomposition from
classical Bergman spaces to Bergman-Orlicz spaces whose growth function belongs to $\mathscr{L}$.
\begin{theorem}\label{mainresult1}
Let $\Phi\in \mathscr{L}_p$ and  $b\in\mathbb R$ with $b>\frac{n+1+\alpha}p$.
There exists a sequence $a =\lbrace a_k\rbrace_{k=1}^{\infty}$
in $\mathbb B^n$, such that $\mathcal A^{\Phi}_\alpha(\mathbb
 B^n)$ consists exactly of functions of the form
 \begin{equation}\label{atomic-expression}
  f(z)=\sum_{k=1}^{\infty}\frac{c_k}{(1-\langle z,a_k\rangle)^b},
  \qquad\quad z\in\mathbb B^n,
 \end{equation}
 where $\lbrace c_k\rbrace_{k=1}^{\infty}$ is a sequence of complex numbers that satisfies the condition
  \begin{equation}
  \sum_{k=1}^{\infty}(1-|a_k|^2)^{n+1+\alpha}\Phi
\left(\frac{|c_k|}{(1-|a_k|^2)^b}\right)<\infty
  \end{equation} and the series converges in
 the norm topology of $\mathcal A^{\Phi}_\alpha(\mathbb
 B^n)$. Moreover, there exists a sequence $\lbrace c_k\rbrace$ such that
 \begin{equation}\label{atom-crit}
  \int_{\mathbb B^n}\Phi(|f(z)|)d\nu_\alpha(z)\simeq \sum_k(1-|a_k|^2)^{n+1+\alpha}\Phi
  \left(\frac{|c_k|}{(1-|a_k|^2)^{b}}\right).
 \end{equation}
\end{theorem}
After some minor modifications, this result is still valid for Bergman-Orlicz spaces with convex growth function.
For $\Phi$ a convex growth function, we recall that the complementary function, $\Psi : [0, \infty) \rightarrow [0, \infty)$, is defined by
\begin{equation}\label{complementarydefinition}
\Psi(s)=\sup_{t\in\mathbb R_+}\{ts - \Phi(t)\}.
\end{equation}
 One easily checks that if $\Phi\in \mathscr{U}$, then $\Psi$ is a growth function of lower type such that the function $t\mapsto \frac {\Psi (t)}t$ is non-decreasing, but which may not  satisfy the $\Delta_2-$condition. We say that the growth function $\Phi$ satisfies the $\bigtriangledown_2-$condition whenever both $\Phi$ and its complementary satisfy the $\Delta_2-$condition.\\

We shall also prove the following analogous of the previous theorem for growth functions belonging to $\mathscr U.$ 
\vskip .2cm
\begin{theorem}\label{mainresult2}
Let $\Phi\in \mathscr{U}$ and  $b\in\mathbb R$ with $b>n+1+\alpha$. We suppose that
$\Phi$ satisfy the $\nabla_2-$condition.
There exists a sequence $a =\lbrace a_k\rbrace_{k=1}^{+\infty}$
in $\mathbb B^n$, such that $\mathcal A^{\Phi}_\alpha(\mathbb
 B^n)$ consists exactly of functions of the form
 \begin{equation}\label{atomic-expression2}
  f(z)=\sum_{k=1}^{\infty}\frac{c_k}{(1-\langle z,a_k\rangle)^b},
  \qquad\quad z\in\mathbb B^n,
 \end{equation}
 where $\lbrace c_k\rbrace_{k=1}^{\infty}$ is a sequence of complex numbers that satisfies the condition
 $$
 \sum_{k=1}^{\infty} (1-|a_k|^2)^{n+1+\alpha}\Phi
\left(\frac{|c_k|}{(1-|a_k|^2)^b}\right)<\infty
 $$ and the series converges in
 the norm topology of $\mathcal A^{\Phi}_\alpha(\mathbb
 B^n)$. Moreover, we have
 \begin{equation}
  \int_{\mathbb B^n}\Phi(|f(z)|)d\nu_\alpha(z)\simeq \sum_k(1-|a_k|^2)^{n+1+\alpha}\Phi
  \left(\frac{|c_k|}{(1-|a_k|^2)^{b}}\right)
 \end{equation}
 when the left hand side is bounded by $1$.
\end{theorem}

It is well-known in the classical case that such atomic decompositions may be used   to obtain weak factorization theorems for
$\mathcal A_{\alpha}^{p}(\mathbb B^n)$ for $p\leq 1,$ in terms of products
of functions in Bergman spaces \cite[Corollary 2.33]{KZ}. Recently, using the above atomic decomposition and
their characterization of boundedness of Hankel operators (with loss) between two Bergman spaces, J. Pau and R. Zhao (\cite{PauZhao2}) extended these weak factorization theorems for
$\mathcal A_{\alpha}^{p}(\mathbb B^n)$ with $p>1$. So, for $p>0$, each function $f\in
\mathcal A_{\alpha}^{p}(\mathbb B^n)$ can be decomposed as
$$ f(z)=\sum_{k}g_k(z)h_k(z), \qquad z\in\mathbb B^n,$$
where each $g_k$ is in $\mathcal A_{\alpha}^{q}(\mathbb B^n)$ and each
$h_k$ is in $\mathcal A_{\alpha}^{r}(\mathbb B^n)$, where $\frac 1p=\frac 1q +\frac 1r$, with
\begin{equation}\label{p-norm}\sum_{k}\|g_k\|_{q,\alpha} \|h_k\|_{r,\alpha}\lesssim \|f\|_{p,\alpha}.\end{equation}
This last inequality can be strengthened for $p\leq 1$ to obtain a weak factorization such that
\begin{equation}\label{p-crit}\sum_{k}\|g_k\|_{q,\alpha}^p \|h_k\|_{r,\alpha}^p\simeq \|f\|_{p,\alpha}^p.\end{equation}
One may ask whether such weak factorizations may be obtained for Bergman-Orlicz spaces. This first proposition is an immediate corollary of an observation on Orlicz spaces given in \cite{VT}.
\begin{proposition} \label{incl-prod}
Let $\Phi_1$ and $\Phi_2$ be two growth functions of finite lower type and let $\Phi$ be a growth function such that
\begin{equation}
\Phi^{-1}=\Phi_1^{-1}\times \Phi_2^{-1}.
\end{equation}
Then the product of two functions that are respectively in $\mathcal A_{\alpha}^{\Phi_1}(\mathbb B^n)$ and $\mathcal A_{\alpha}^{\Phi_2}(\mathbb B^n)$ is in $\mathcal A_{\alpha}^{\Phi}(\mathbb B^n)$. Moreover
$$\|fg\|_{\Phi, \alpha}\lesssim \|f\|_{\Phi_1, \alpha}\|g\|_{\Phi_2, \alpha}.$$
Here for a growth function $\Phi$, $\Phi^{-1}$ is the inverse function of $\Phi$.
\end{proposition}
One may ask whether one has a weak factorization of $\mathcal A_{\alpha}^{\Phi}(\mathbb B^n)$ in this context. Using the atomic decomposition obtained here and natural factorization
together with good estimates that can be found in \cite{ST2}, we shall obtain for $\Phi\in \mathscr{L}$, weak factorization theorems for
$\mathcal A_{\alpha}^{\Phi}(\mathbb B^n)$ in terms of products of functions in Bergman-Orlicz spaces. It is done in the last section of this paper. But we do not succeed in giving a critical equivalent of the norm, as in \eqref{p-crit} 

 In view of applications to Hankel
operators studied in \cite{ST2}, our second interest here
is to obtain another type of weak factorization for functions in
$\mathcal A_{\alpha}^{\Psi}(\mathbb B^n)$,
$\Psi\in \mathscr{L}_p$, in terms of products of functions in
$\mathcal A_{\alpha}^{\Phi}(\mathbb B^n)$ and in the Bloch space.
We recall
 that given an holomorphic function $f$ on $\mathbb B^n$, the
 radial derivative $Rf$ of $f$ is defined by
 $$Rf(z)=\sum_{j=1}^{n}z_j \frac{\partial f}{\partial z_j}(z).$$
 The Bloch class $\mathcal B$ is the space
of holomorphic functions in $\mathbb B^n$ such that
$$\sup_{z\in\mathbb B^n}|Rf(z)|(1-|z|^2) <\infty.$$
The norm on $\mathcal B$ is given by $\|f\|_{\mathcal B}=|f(0)|+
\sup_{z\in\mathbb B^n}|Rf(z)|(1-|z|^2) $.
One has the following proposition for products  of functions that are respectively in
$\mathcal A_{\alpha}^{\Phi}(\mathbb B^n)$ and in $\mathcal B$.
\begin{proposition}\label{incl-bloch}
Let $\Phi$ be a growth function of finite lower type  (resp. of finite upper type). The product maps continuously
$\mathcal A^{\Phi}_\alpha(\mathbb B^n)\times \mathcal B$ into
$\mathcal A^{\Psi}_\alpha(\mathbb B^n)$, where
$\Psi(t)=\Phi\left(\frac{t}{\log(e+t)}\right)$. Moreover,
$$\|fg\|_{\Psi, \alpha}\lesssim \|f\|_{\Phi, \alpha}\|g\|_{\mathcal B}.$$

\end{proposition}

The following is our second main result.
\begin{theorem}\label{mainresult3}
 Let $\Phi\in \mathscr{L}_p$ and let $\Psi(t)=\Phi\left(\frac{t}{\log(e+t)}\right)$. Every function $f\in \mathcal A^{\Psi}_\alpha(\mathbb
 B^n)$  may be written as the sum
 $$f=\sum_{k=1}^{+\infty}f_kb_k,
 $$
 with $f_k\in\mathcal A^{\Phi}_\alpha(\mathbb
 B^n)$ and $b_k\in\mathcal B,$ with
 \begin{equation}\label{inequality-lux}
   \sum_{k=1}^{+\infty}\|f_k\|_{\Phi, \alpha}\|b_k\|_{\mathcal B}\lesssim  \|f\|_{\Psi, \alpha}.
\end{equation}
Moreover, if $\|f\|_{\Psi, \alpha}\leq 1$,
\begin{multline} \label{critical}
 \int_{\mathbb B^n}\Psi(|f(z)|)d\nu_\alpha(z)\simeq \sum_{k=1}^{+\infty}\int_{\mathbb B^n}\Psi(|f_k(z)b_k(z)|)d\nu_\alpha(z)\\
\simeq \sum_{k=1}^{+\infty}\int_{\mathbb B^n}\Phi(|f_k(z)|)d\nu_\alpha(z) \times\,\|b_k\|_{\mathcal B}.
\end{multline}
\end{theorem}
These two estimates can be considered as the equivalent, in this context, of \eqref{p-crit} and \eqref{p-norm}. Remark that, except when $\Phi$ is equivalent to a homogeneous function, there is no way to pass from the Luxembourg norm to the quantity $\int \Phi(|\cdot|)d\nu_\alpha$. In the previous statement only one of the two Bergman-Orlicz spaces involved can coincide with some $A^p_\alpha(\mathbb{B}^n)$.
\smallskip

The same kind of statement has been considered for Hardy-Orlicz spaces and the class $BMOA$ in \cite{BG}. But only the equivalent of \eqref{inequality-lux} has been obtained. There is no equivalence as in the previous theorem. We have a better understanding of weak factorization in the context of Bergman-Orlicz spaces.

\vskip .2cm

The paper is organized as follows. In section \ref{section2},
we collect and establish some results that will be used later.
In section \ref{section3}, we give proofs of atomic decomposition theorems for functions in Bergman-Orlicz spaces. In particular, we establish Thorem 1.2 and Theorem 1.3. In section 4, we first prove Proposition \ref{incl-bloch}; next,
we prove weak factorization theorems for Bergman-Orlicz spaces, the first in terms of  products of two factors, one in the Bloch space, the other in a Bergman-Orlicz space (Theorem \ref{mainresult3}), and the second in terms of products of two factors in two Bergman-Orlicz spaces (Theorem \ref{mainresult4}). We apply the first weak factorization theorem to recover a characterization result \cite{ST2} of bounded small Hankel operators from a Bergman-Orlicz space $\mathcal A^{\Phi}_\alpha(\mathbb B^n)$ to the weighted Bergman spaces $\mathcal A^1_\alpha (\mathbb B^n).$ 
\vskip .2cm
 Finally, all over the text, $C$ will be a constant not necessary
 the same at each occurrence. We will also use the notation $C(k)$
 to express the fact that the constant depends on the underlined
 parameter $k$. Given two positive quantities $A$ and $B$, the notation
 $A\lesssim B$ means that $A\le CB$ for some positive uniform constant $C$.
 When $A\lesssim B$ and $B\lesssim A$, we write $A\simeq B$.

\section{Preliminaries}\label{section2}
In this section, we recall some known results and establish some estimates that are needed in our study.

\subsection{Some geometric properties in the unit ball}
We recall the following facts for which details can be found in \cite{KZ}.

For $z\in\mathbb B^n$, let $\varphi_z$ be the involutive automorphism of $\mathbb B^n$
that interchanges $z$ and $0$. That is, $\varphi_z$ is a holomorphic function from $\mathbb B^n$
to itself that satisfies $\varphi_z\circ\varphi_z=id$ and $\varphi_z(0)=z$ and $\varphi_z(z)=0$. Using
the map $\varphi_z$, the Bergman metric, $d$ on $\mathbb B^n$, is defined by
$$d(z,w)=\frac 12\log\left(\frac{1+|\varphi_z(w)|}{1-|\varphi_z(w)|}\right).$$
For $r>0$, we denote by $D(z,r)$ the Bergman ball, that is the ball with respect to the Bergman
metric, of radius $r$ and centered at $z$. It is well-known that for $w\in D(z,r)$
\begin{equation}
 \nu_\alpha(D(z,r))\simeq |1-\langle z,w\rangle|^{n+1+\alpha}\simeq \left(1-|z|^2\right)^{n+1+\alpha}
 \simeq \left(1-|w|^2\right)^{n+1+\alpha}.
\end{equation}
Here constants are uniform in $z$.

A sequence $\lbrace a_k\rbrace$ of points in $\mathbb B^n$ is a separated sequence (in Bergman
metric) if there exists a positive constant $\delta>0$ such that $d(a_k,a_j)\geq\delta$ for
any $k\neq j$. A maximal $\delta-$ separated sequence $\lbrace a_k\rbrace$ in $\mathbb B^n$ has the property that
 \begin{enumerate}[\upshape (i)]
\item $\mathbb B^n=\cup_kD(a_k,\delta)$,
\item The sets $D(a_k, \frac{\delta}2)$ are mutually disjoint,
\item Each point $z\in\mathbb B^n$ belongs to at most $N$ of the sets $D(a_k,2\delta)$.
\end{enumerate}

 Here $N$ is an absolute constant, which does not depend of the sequence $\lbrace a_k\rbrace$. A sequence $\lbrace a_k\rbrace_{k=1}^\infty$ satisfying these conditions  is called
an $\delta-$lattice. 

The following lemma will be useful.

\begin{lemma}\label{madedisjoint}\cite[Lemma 2.28]{KZ}
Let $\lbrace a_k\rbrace_{k=1}^\infty$ be a $\delta-$lattice. There is a sequence of  Borel sets $\lbrace D_k\rbrace_{k=1}^\infty$ in $\mathbb B_n$ satisfying the following conditions.
\begin{enumerate}[\upshape (i)]
\item $D(a_k, \frac \delta 4) \subset D_k \subset D(a_k, \delta),$
\item The sets $D_k$ are mutually disjoint,
\item $\mathbb B^n=\cup_k D_k.$ 
\end{enumerate}
\end{lemma}

It is classical that one can jointly construct one $1$-lattice $\lbrace a_k\rbrace$ and one $\eta-$ lattice $\{a_{kj}\}$ with remarkable properties. We have   the following lemma (see  \cite{KZ} for more details).
\begin{lemma}
\label{decomposition-ball}
Let $\eta\in (0, 1)$ be small . There exists an integer $J$, which depends only on $\eta$, such that one can find simultaneously a $1-$lattice $\{a_k\}_{k=1}^\infty$ and an $\eta-$lattice $\{a_{kj}\} $ with $k$ varying from $1$ to $\infty$ and $j$ from $1$ to $J$ with the supplementary property that
$$D(a_k, 1)\subset\cup_j D(a_{kj}, 2\eta).$$
Moreover, if  $\{D_{kj}\}$ (resp. $\{D_k\})$ denotes the sequence of disjoint Borel sets corresponding to the $\eta-$lattice $\{z_{kj}\}$ (resp. to the $1-$lattice $\{z_{k}\})$ as described in Lemma \ref{madedisjoint}, we have 
$$D_k=\cup_{j=1}^J D_{kj}.$$
\end{lemma}
With these notations, for  $b>n$ and let $\beta=b-n-1$, we define the following  operator $S$ by
  \begin{equation}
  \label{operator-S}
  Sf(z)=\sum_{k=1}^{\infty}\sum_{j=1}^J\frac{\nu_\beta(D_{kj})\;  f(a_{kj})}
  {(1-\langle z,a_{kj}\rangle)^b},
 \end{equation}
for $f\in L^1(\mathbb B^n, d\nu_\beta(z))$.

We will need the following lemma concerning $S$.
 \begin{lemma}\cite[Lemma 2.29]{KZ}
 \label{key-estimate-decomposition}
  Let  $\{a_k\}_{k=1}^\infty$ and $\{a_{kj}\} $ be as in Lemma \ref{decomposition-ball}. For any $s>0$ and $\alpha>-1$ there exists a constant $C>0$, independent
of the separation constant $\eta$, such that
$$|f(z)-Sf(z)|\leq C\sigma\sum_{k=1}^{+\infty}
\frac{(1-|a_k|^2)^{b}}{|1-\langle z, a_k\rangle|^b}\left(
\int_{D(a_k,2)}|f(w)|^s \frac {d\nu_\alpha(w)}{\nu_\alpha(D(a_k, 2))}\right )^{\frac 1s}$$
for all $ z\in\mathbb B^n,$ and $f\in \mathcal A^1_{\beta}(\mathbb B^n)$, where
$\sigma$ depends only on $\eta$ and $\sigma \rightarrow 0$ as $\eta\rightarrow 0$.
\end{lemma}

\subsection{Some useful estimates}
We collect in this subsection some properties of growth functions we
shall use later and establish some useful estimates.

For $\Phi$ a $\mathcal C^1$ growth function, the lower and the upper indices of $\Phi,$ defined in Remark \ref{rmk}, are respectively given by
$$a_\Phi:=\inf_{t>0}\frac{t\Phi^\prime(t)}{\Phi(t)}\,\,\,\textrm{and}\,\,\,b_\Phi:=\sup_{t>0}\frac{t\Phi^\prime(t)}{\Phi(t)}.$$
We recall that when $\Phi$ is convex, then $1\le a_\Phi\le b_\Phi<\infty$ and, if $\Phi$ is concave, then $0<a_\Phi\le b_\Phi\le 1$. We have the following simple but useful fact \cite{ST2}.
\begin{lemma}\label{indices}
Let $\Phi$ be a $\mathcal C^1$ growth function. Denote by $p$ and $q$ its lower and its upper indices respectively. Then the functions $\frac{\Phi(t)}{t^p}$ and $\frac{\Phi^{-1}(t)}{t^{1/q}}$ are non-decreasing.
\end{lemma}

\begin{remark}\label{rem:about-indices}
One useful way to use Lemma \ref{indices} is to observe that it implies the following: if $\Phi\in \mathscr{L}_p$ for some $p\in (0, 1)$, then the growth function $\Phi_p$, defined by
$\Phi_p(t)=\Phi(t^{1/p})$, is in $\mathscr{U}^q$ for some $q\geq 1$ (e.g. $q=\frac 1p$.) So we may assume that $\Phi_p$ is convex.
\end{remark}
We will make use very often of the following classical estimate.
\begin{theorem}\cite[Theorem 1.12]{KZ}
\label{thm:estimate-classic}
Let $\alpha>-1$ and $c>0$. The following integral
$$J_{c,\alpha}(z)=\int_{\mathbb B^n}\frac{(1-|w|^2)^\alpha d\nu(w)}{|1-\langle z,w\rangle|^{n+1+\alpha+c}},\qquad\quad z\in\mathbb B^n,$$
have the following asymptotic property.
\begin{equation}
    J_{c,\alpha}(z)\simeq (1-|z|^2)^{-c}, \quad |z|\rightarrow 1^-.
\end{equation}
\end{theorem}
We use in particular this theorem for computations on atoms.

\begin{lemma}
\label{lem:norm-estimate1}
Let $\Phi\in \mathscr{L}_p$, $a\in\mathbb B^n$ and a real $b$ such that
$b>\frac{n+1+\alpha}{p}$. There exists a positive
constant $C$  such that for all $\lambda\in \mathbb C$  and all functions $f$ of the form 
$f(z)=\frac{\lambda}{(1-\langle z, a\rangle)^b}$. There exists a positive
constant $C$ (depending only on fixed constants) such that:
\begin{multline}
\label{eq:estimate1}
C^{-1}\Phi\left(\frac{|\lambda|}{(1-|a|^2)^{b}}\right)(1-|a|^2)^{n+1+\alpha}
\leq \int_{\mathbb B^n} \Phi(|f(z)|)d\nu_\alpha(z) \\
\leq C
\Phi\left(\frac{|\lambda|}{(1-|a|^2)^{b}}\right)(1-|a|^2)^{n+1+\alpha}.
\end{multline}
\end{lemma}
\begin{proof}
Recall that $\Phi$ is of lower type $p$, so that $\Phi(st)\leq Ct^p\Phi (s)$ for $0<s\leq 1$. Moreover, since $\Phi$ satisfies the $\Delta_2$ condition, such an inequality is also valid for $1\leq t\leq 2^b$ (eventually after a modification of the constant). We can use it here for $t=\left (\frac{1-|a|^2}{|1-\langle z,a\rangle|}\right )^b$, which is bounded by $2^b$. We write that
\begin{eqnarray}\label{eq:41}
 \int_{\mathbb B^n}\Phi(|f(z)|)d\nu_\alpha(z)
 &\leq & C\Phi\left(|\lambda|(1-|a|^2)^{-b}\right)\int_{\mathbb B^n}
 \frac{(1-|a|^2)^{bp}}{|1-\langle z,a\rangle|^{bp}}d\nu_\alpha(z) \nonumber\\
 &\leq & C \Phi\left(\frac{|\lambda|}{(1-|a|^2)^{b}}\right)(1-|a|^2)^{n+1+\alpha},
\end{eqnarray}
where we used Theorem \ref{thm:estimate-classic} in the last line.  We have obtained the upper bound of (\ref{eq:estimate1}).
\smallskip

To obtain the lower bound we fix a positive real $r$. Using the fact that in the
Bergman ball $D(a,r)$, we have
$$(1-|a|^2)\simeq (1-|z|^2)\simeq |1-\langle z,a\rangle| \quad \quad \quad \left (z\in D(a,r)\right ),$$
we obtain:
 \begin{eqnarray}\label{eq:lowerbound}
 \int_{\mathbb B^n}\Phi(|f(z)|)d\nu_\alpha(z)
 &\geq &  \int_{D(a,r)}\Phi(|f(z)|)d\nu_\alpha(z)\nonumber\\
 &\geq & C\Phi\left(|\lambda|(1-|a|^2)^{-b}\right)\int_{D(a,r)}
 d\nu_\alpha(z)\nonumber\\
 &\simeq & C \Phi\left(\frac{|\lambda|}{(1-|a|^2)^{b}}\right)(1-|a|^2)^{n+1+\alpha}.
\end{eqnarray}
 \end{proof}
\begin{remark}
\label{rem:norme-estimate2}
Under mild modifications, if $\Phi\in \mathscr{U}^q$, the estimate in (\ref{eq:estimate1}) still holds for functions in $\mathcal A^\Phi_\alpha (\mathbb B^n)$, provided that $b>n+1+\alpha.$

A consequence of Lemma \ref{lem:norm-estimate1} gives the Luxembourg quasi-norm estimate for such $f$. We have
\begin{equation}\label{eq:luxembourg-estimate}
\|f\|_{\Phi, \alpha} \simeq \frac{|\lambda|}{(1-|a|^2)^{b}\Phi^{-1}\left(\frac 1{(1-|a|^2)^{n+1+\alpha}}\right)}
\end{equation}
\begin{definition}
Functions $f$ of the form $f(z)=\frac{\lambda}{(1-\langle z, a\rangle)^b}$ are called (non normalized) atoms of $\mathcal A^\Phi_\alpha (\mathbb B^n)$.
\end{definition}

\end{remark}

To finish this subsection, we recall elementary properties of norms and integrals. Remark first that the equivalence
\begin{equation}
\int_{\mathbb B^n} \Phi (|f(z)|)d\nu_\alpha(z)\lesssim 1 \quad  \mbox{ is equivalent to} \quad \|f\|_{\Phi, \alpha}\lesssim 1.
\end{equation}
Moreover, when this condition is satisfied,
\begin{eqnarray}
\|f\|_{\Phi, \alpha}&\lesssim \int_{\mathbb B^n} \Phi (|f(z)|)d\nu_\alpha(z)\lesssim \|f\|_{\Phi, \alpha}^p \qquad &\mbox{for}\; \Phi\in \mathscr L_p \label{upPhi} \\
\|f\|_{\Phi, \alpha}^q&\lesssim \int_{\mathbb B^n} \Phi (|f(z)|)d\nu_\alpha(z) \lesssim  \|f\|_{\Phi, \alpha} \qquad &\mbox{for}\; \Phi\in \mathscr U^q. \label{upLux}
\end{eqnarray}
One cannot reverse these inequalities.

We shall use the following results about Luxembourg norm estimates for bounded functions in $\mathcal A_{\alpha}^{\Phi}(\mathbb B^n)$. These are easy extensions of the same type of results in \cite[Lemma 3.9]{LLQR}.
\begin{lemma}\label{lem:bounded-concave}
Let $\alpha>-1$ and $\Phi\in \mathscr{L}_p$. For any bounded holomorphic function $f$ in $\mathbb B^n$, one has:
\begin{equation}\label{eq:bounded-concave}
\|f\|_{\Phi, \alpha}\leq \frac{\|f\|_\infty}{\Phi^{-1}\left(\frac {\|f\|_\infty^p}{\|f\|_{p,\alpha}^p}\right)}.
\end{equation}
\end{lemma}

\begin{lemma}\label{lem:bounded-convex}
Let $\alpha>-1$ and $\Phi\in \mathscr{U}^q$.  For any bounded holomorphic function $f$ in $\mathbb B^n$, one has:
\begin{equation}\label{eq:bounded-convex1}
\|f\|_{\Phi, \alpha}\leq \frac{\|f\|_\infty}{\Phi^{-1}\left(\frac {\|f\|_\infty}{\|f\|_{1,\alpha}}\right)}.
\end{equation}
\end{lemma}

\section{Atomic decomposition for Bergman-Orlicz spaces}
\label{section3}
In this section, we give the proofs of Theorem \ref{mainresult1} and Theorem \ref{mainresult2}. Our proofs are adapted from the proofs in the classical weighted Bergman spaces.
\subsection{Proof of Theorem \ref{mainresult1}}
This subsection is devoted to the proof of Theorem \ref{mainresult1}.
Let $\Phi\in \mathscr{L}_p$. In one direction we have more than what is stated in the theorem. Namely, in the next proposition, $\lbrace a_k\rbrace$ is an arbitrary sequence of points of $\mathbb B^n.$
\begin{proposition}\label{above}
Suppose that $b>\frac {n+1+\alpha}p.$ Let $\Phi\in \mathscr{L}_p$ and let $\lbrace a_k\rbrace$ be a sequence of points in $\mathbb B^n$. Assume that $\lbrace c_k\rbrace$ satisfies the condition
$$\sum_k (1-|a_k|^2)^{n+1+\alpha}\Phi \left(\frac{|c_k|}{(1-|a_k|^2)^b}\right)<\infty.$$
 Then the series $\sum_k \frac{c_k}{(1-\langle z, a_k\rangle)^b}$ converges in $\mathcal A^\Phi_\alpha (\mathbb B^n)$ to a function $f$ and
 $$ \int_{\mathbb B^n}\Phi \left (\sum_k \left \vert \frac{c_k}{(1-\langle z, a_k\rangle)^b} \right \vert \right) d\nu_\alpha(z)\lesssim  \sum_{k=1}^{+\infty}(1-|a_k|^2)^{n+1+\alpha}\Phi
\left(\frac{|c_k|}{(1-|a_k|^2)^b}\right).$$
\end{proposition}
\begin{proof}
Let us define
$f_k(z)= \frac{c_k}{(1-\langle z, a_k\rangle)^b}$.
We use the concavity of the function $\Phi$ and Lemma \ref{lem:norm-estimate1} to obtain that for all finite sets of indices $K$ one has
\begin{eqnarray*}
\int_{\mathbb B^n}\Phi(|\sum_{k\in K} f_k|)d\nu_\alpha&\leq &\sum_{k\in K}\int_{\mathbb B^n}\Phi(| f_k|)d\nu_\alpha\\
&\leq& C  \sum _{k\in K} (1-|a_k|^2)^{n+1+\alpha}\Phi
\left(\frac{|c_k|}{(1-|a_k|^2)^b}\right).
\end{eqnarray*}
 The space $\mathcal A^\Phi_\alpha (\mathbb B^n)$ is a complete metric space for the distance defined by
$$(f,g)\mapsto \int_{\mathbb B^n} \Phi(|f(z)-g(z)|)d\nu_\alpha(z),$$
or, equivalently for the one defined by the Luxembourg quasi-norm. The condition on the sequence $\lbrace c_k\rbrace$ implies that the sequence of partial sums of the series is a Cauchy sequence in the space $\mathcal A^\Phi_\alpha (\mathbb B^n)$. The function $f=\sum_{k=1}^\infty f_k$ is its limit. We conclude at once.
\end{proof}

Let  $\{a_k\}_{k=1}^\infty$ and  $\{a_{kj}\}$ be as in Lemma \ref{decomposition-ball}. The latter sequence is the sequence for which we will prove the representation of Theorem \ref{mainresult1}. For better understanding we keep a double index. We show now that every function $f\in\mathcal A^\Phi_\alpha (\mathbb B^n)$ may be written as in  (\ref{atomic-expression}), that is,
$$f(z)=\ \sum_{k=1}^{+\infty}\sum_{j=1}^{J}\frac{c_{kj}}{(1-\langle z, a_{kj}\rangle)^b}.$$  The constant $\eta$ will be chosen sufficiently small later on. We first prove that $f-Sf$ is small in $\mathcal A^\Phi_\alpha (\mathbb B^n)$ when $\eta$ is small enough, where $S$ is defined in (\ref{operator-S}). Remark that, since $\mathcal A^\Phi_\alpha (\mathbb B^n)$ is continuously embedded in  $\mathcal A^p_\alpha(\mathbb B^n)$, for $f\in\mathcal A^\Phi_\alpha (\mathbb B^n)$ the inequality
$$|f(z)|^p\nu_\alpha (D(z, 1))\leq C\int_{\mathbb B^n}\Phi(|f(z)|)d\nu_\alpha(z)$$
   implies easily that $f$ belongs to $\mathcal A^1_\beta(\mathbb B^n)$ when $b>\frac{n+1+\alpha}{p}$ and $\beta=b-n-1$ as above.
 From Lemma \ref{key-estimate-decomposition}, where we take $s=p$, the fact that $\Phi$ is of lower type $p$ and Proposition \ref{above}, there exists $C>0$ such that
\begin{multline*}
\int_{\mathbb B^n}\Phi\left(|f(z)-Sf(z)|\right)d\nu_\alpha(z)
\leq \\
 C\sigma^p \int_{\mathbb B^n}\Phi\left(\sum_{k=1}^{+\infty}
\frac{(1-|a_k|^2)^{b}}{|1-\langle z, a_k\rangle|^b}\left(
\int_{D(a_k,2)}|f(w)|^p \frac{d\nu_\alpha(w)}{\nu_\alpha(D(a_k, 2))}\right)^{\frac 1p}\right)d\nu_{\alpha}(z)\\
\leq  C\sigma^p \sum_{k=1}^{+\infty}(1-|a_k|^2)^{n+1+\alpha}\Phi\left\lbrace  \left(\int_{D(a_k,2)}|f(w)|^p \frac{d\nu_\alpha(w)}{\nu_\alpha(D(a_k, 2))}\right)^{\frac 1p} \right\rbrace.
\end{multline*}
For the first inequality, we took $\eta$ sufficiently small so that $C\eta \leq 1.$ Since $\Phi_p(t)=\Phi\left(t^{1/p}\right)$ is convex (see Remark \ref{rem:about-indices}), we will make use of the following Jensen inequality
$$\Psi\left(\int_Xgd\mu\right)\leq\int_X\Psi(g)d\mu,$$
valid  for any convex function $\Psi$, nonnegative function $g$, and a probability measure $d\mu$ on $X$, to obtain
\begin{eqnarray*}
&&\int_{\mathbb B^n}\Phi\left(|f(z)-Sf(z)|\right)d\nu_\alpha(z)\\
 &\leq&\sigma^p \sum_{k=1}^{+\infty}(1-|a_k|^2)^{n+1+\alpha}\Phi_p\left(\int_{D(a_k,2)}|f(w)|^p\frac{d\nu_\alpha(w)}{\nu_\alpha(D(a_k,2))}\right)\\
&\leq & C\sigma^p \sum_{k=1}^{+\infty}(1-|a_k|^2)^{n+1+\alpha}  \int_{D(a_k,2)}\Phi_p\left(|f(w)|^p\right)\frac{d\nu_\alpha(w)}{\nu_\alpha(D(a_k,2))}\\
&\leq &  C\sigma^p \sum_{k=1}^{+\infty}  \int_{D(a_k,2)}\Phi\left(|f(w)|\right)d\nu_\alpha(w).
\end{eqnarray*}
We have used the fact that $\nu_\alpha(D(a_k,2))\simeq (1-|a_k|^2)^{n+1+\alpha}.$
 
By the finite overlapping property of a $1-$lattice (property  (iii)) we have finally
$$\int_{\mathbb B^n}\Phi\left(|f(z)-Sf(z)|\right)d\nu_\alpha(z)
\leq  CN\sigma^p\int_{\mathbb B^n} \Phi\left(|f(w)|\right)d\nu_\alpha(w).
$$
We choose $\eta$ small enough so that $CN\sigma^p\leq \frac 12$.

For $g\in \mathcal A^\Phi_\alpha (\mathbb B^n)$, we deduce from the previous inequality that $\sum_{n\geq 0} \int\Phi(|(I-S)^n g|)d\nu_\alpha\leq 2 \int\Phi(| g|)d\nu_\alpha$. We use again the concavity of $\Phi$ to deduce that  the Neumann series
 $\sum_{n=0}^\infty (I-S)^n g$ converges in $\mathcal A^\Phi_\alpha (\mathbb B^n)$. As for Banach spaces, we obtain that the bounded operator $S$ on $\mathcal A^\Phi_\alpha (\mathbb B^n)$ is invertible and its inverse $S^{-1}$ is given by $S^{-1} (g) = \sum_{n=0}^\infty (I-S)^n g.$ Therefore, every $f\in \mathcal A^\Phi_\alpha (\mathbb B^n)$ admits a representation 
$$f(z)=\sum_{k=1}^{\infty}\sum_{j=1}^J\frac{c_{kj}}{(1-\langle z,a_{kj}\rangle)^b},
 $$
where $$c_{kj}=\nu_\beta(D_{kj})g(a_{kj})\qquad\textrm{and}\qquad g=S^{-1}f\in \mathcal A^\Phi_\alpha (\mathbb B^n).$$
It  remains to show that
$$I:= \sum_{k=1}^{\infty}\sum_{j=1}^J (1-|a_{kj}|^2)^{n+1+\alpha}\Phi
\left(\frac{|c_{kj}|}{(1-|a_{kj}|^2)^b}\right)\lesssim\int_{\mathbb B^n}\Phi(|f(z)|)d\nu_\alpha(z).$$
We know that
\begin{equation}
\label{eq:311}
    \nu_{\beta}(D_{kj})\leq \nu_{\beta}(D_{k})\simeq (1-|a_{k}|^2)^{n+1+\beta}=(1-|a_{k}|^2)^{b},
\end{equation}
and, by the mean value property \cite[Lemma 2.24]{KZ}, we also have
\begin{equation}
    \label{eq:312}
|g(a)|^p\leq C \frac1{\nu_\alpha(D(a,1)}\int_{D(a,1)}|g(w)|^p d\nu_\alpha(w).
\end{equation}
Using (\ref{eq:311}) and (\ref{eq:312}), the Jensen inequality as above and the finite overlapping property, we have
\begin{eqnarray*}
   I
&\lesssim & \sum_{k=1}^{+\infty}\sum_{j=1}^J (1-|a_{kj}|^2)^{n+1+\alpha}\Phi\left(|g(a_{kj})|\right)\\
&\lesssim & \sum_{k=1}^{+\infty}\sum_{j=1}^J (1-|a_{kj}|^2)^{n+1+\alpha}\Phi_p\left( \int_{D(a_{kj},1)}|g(w)|^p \frac{d\nu_\alpha(w)}{\nu_\alpha (D(a_{kj}, 1)}\right)\\
&\leq & CJ\sum_{k=1}^{+\infty} (1-|a_{k}|^2)^{n+1+\alpha} \int_{D(a_{k},1)}\Phi(|g(w)|) \frac{d\nu_\alpha(w)}{\nu_\alpha(D(a_{k},1))}\\
&\leq & CJN \int_{\mathbb B^n}\Phi(|g(w)|)d\nu_\alpha(w)<\infty,
\end{eqnarray*}
which is what we wanted to prove. The converse inequality has been given by Proposition \ref{above}. So we have completed the proof of Theorem
\ref{mainresult1}.

\medskip

 We finish this subsection by a remark. Theorem \ref{mainresult1} gives the integral $\int \Phi(|f|)d\nu_\alpha$ in terms of the coefficients of a representation of $f$. In order to deal with Luxembourg norms, we  first give some definitions. Given a growth function
$\Phi$ and $b>0$, 
we define the space $l^\Phi_{\alpha,b}$ as
the space of couple of sequences $\lbrace a_k\rbrace_{k\in\mathbb N}$
in $\mathbb B^n$,  and $\lbrace c_k\rbrace_{k\in\mathbb N}$ in $\mathbb C$ such that, for some $\lambda>0$,
\begin{equation*}
\sum_k(1-|a_k|^2)^{n+1+\alpha}\Phi\left(\frac {\frac{|c_k|}\lambda}{(1-|a_k|^2)^b}\right)<\infty.
\end{equation*}
An element of this space identifies with a sequence of non normalized atoms of the form $\{f_{a_k, c_k}\}_k$, with
\begin{equation} \label{def-atoms}
f_{a_k, c_k}=\frac{c_k}{(1-\langle z, a_k\rangle)^b}.
\end{equation}
We define the quasi norm on this space by:
\begin{equation}
\label{eq:norm-sequence}
    \|\{f_{a_k, c_k}\}\|_{l^\Phi_{\alpha,b}}= \inf\{\lambda >0: \sum_k(1-|a_k|^2)^{n+1+\alpha}\Phi
\left(\frac {\frac{|c_k|}\lambda}{(1-|a_k|^2)^b}\right)\leq 1\}.
\end{equation}

With this definition it is straigthforward to deduce from Theorem \ref{mainresult1} that
\begin{equation}\label{in-norm}
\|f\|_{\Phi, \alpha}\simeq  \inf_{\lbrace a_k\rbrace,\lbrace c_k\rbrace}\left\lbrace \|\lbrace f_{a_k,c_k}\rbrace\|_{l^{\Phi}_{\alpha, b}},\quad  \; f=\sum f_{a_k, c_k}\right\rbrace.
\end{equation}


\bigskip

\subsection{Proof of Theorem \ref{mainresult2}}
This subsection is devoted to the proof of Theorem \ref{mainresult2}. Proposition  \ref{above} is no more valid in all generality, but we have the following proposition.

\begin{proposition}\label{above-convex}
Let $\Phi\in  \mathscr{U}^q$ and let $\lbrace a_k\rbrace_{k=1}^\infty$ be a sequence of $r$-separated points in $\mathbb B^n$. Assume that $\lbrace c_k\rbrace_{k=1}^\infty$ satisfies the condition

 \begin{equation}\label{eq:bounded-1}
     \sum_{k=1}^\infty (1-|a_k|^2)^{n+1+\alpha}\Phi
\left(\frac{|c_k|}{(1-|a_k|^2)^b}\right)\leq 1.
\end{equation}
 
 Then the series $\sum_{k=1}^\infty \frac{c_k}{(1-\langle z, a_k\rangle)^b}$ converges in $\mathcal A^\Phi_\alpha (\mathbb B^n)$ to a function $f$, which is such that
 $\int\Phi(|f|)d\nu_\alpha\lesssim 1$.
Furthermore, for $\lbrace a_k\rbrace$ a given $r-$separated sequence and  $\lbrace c_k\rbrace$ such that $\{f_{a_k, c_k}\}$ is in $l^\Phi_{\alpha, b}$, the series $\sum f_{a_k,c_k}$ converges in $\mathcal A^\Phi_\alpha (\mathbb B^n)$ and
\begin{equation}\label{ineq-lux}\left\|\sum f_{a_k,c_k}\right\|_{\Phi, \alpha}\lesssim \|\{f_{a_k,c_k}\}\|_{l^\Phi_{\alpha, b}}.\end{equation}
\end{proposition}
\begin{proof}
Let $\Phi\in \mathscr{U}^q$. We assume that $\Phi$ is convex, so that the Luxembourg norm is a norm. We will prove a little more, that is,
\begin{equation}\label{abs}\| \sum_{k=1}^\infty\frac{|c_k|}{|1-\langle z,a_k\rangle|^b}\|_{\Phi, \alpha}\lesssim 1\end{equation}
assuming that (\ref{eq:bounded-1}) holds. Let us assume that we succeeded in proving this. Then, by \eqref{upLux}, the same inequality holds for $\int_{\mathbb B^n} \Phi(\sum_{k=1}^\infty\frac{|c_k|}{|1-\langle z,a_k\rangle|^b})d\nu_\alpha(z)$. As a consequence,
$$\int_{\mathbb B^n}\Phi\left(\sum_{k=N}^\infty\frac{|c_k|}{|1-\langle z,a_k\rangle|^b}\right)d\nu_\alpha(z) $$ tends to $0$ when $N$ tends to $\infty$ and the same is valid for the Luxembourg norm, so that the sequence of partial sums of the series $\sum f_{a_k,c_k}$ is a Cauchy sequence, which converges and its sum satisfies the required estimate.

So let us prove \eqref{abs}.
We will make use of the operator
 \begin{equation}
  \label{operator-T}
  Tf(z)=\int_{\mathbb B^n}\frac{f(w)}{|1-\langle z,w\rangle|^b}d\nu_\beta (w).
 \end{equation} 
(recall that $\beta=b-n-1).$ Since $\Phi$ satisfies the $\nabla_2-$condition, the operator $T$, defined in (\ref{operator-T}), is bounded on $L^\Phi_\alpha (\mathbb B^n)$ (see \cite{DHZZ}). We also define the function $F$ as
\begin{equation}
    \label{eq:function-upper-case}
F(z)=\sum_{k=1}^{+\infty} |c_k|(1-|a_k|^2)^{-b}\chi_{D(a_k, r/2)}.
\end{equation}
The balls $D(a_k, r/2)$ are disjoint because of the assumption that the sequence $\{a_k\}$ is $r-$ separated and so
\begin{eqnarray*}
\int_{\mathbb B^n}\Phi\left(|F(z)|\right)d\nu_\alpha(z) & = &\sum_{k=1}^{+\infty}\Phi\left(|c_k|(1-|a_k|^2)^{ -b}\right)\nu_\alpha\left(D(a_k, r/2))\right)\\
&\lesssim &1.
\end{eqnarray*}
This shows that $F\in L^\Phi_\alpha (\mathbb B^n)$. Moreover, because of \eqref{upLux},  we have that
$\|F\|_{\Phi, \alpha}\lesssim  1$.
 Applying $T$ to $F$, we obtain
$$TF(z)=\sum_{k=1}^{+\infty} |c_k|(1-|a_k|^2)^{-b}\int_{D(a_k, r/2)}\frac{1}{|1-\langle z,w\rangle|^b}d\nu_{\beta} (w).$$
Since for $w\in D(a_k, r/2)$, we have $1-|a_k|^2\simeq 1-|w|^2$ and $|1-\langle z,a_k\rangle|\simeq |1-\langle z,w\rangle|$,
$$ \sum_{k=1}^{+\infty} \frac{|c_k|}{|1-\langle z,a_k\rangle|^b} \leq C\,TF(z),\qquad\quad z\in \mathbb B^n$$ Using the continuity properties of $T$ in the Banach space $L^\Phi_\alpha (\mathbb B^n)$, we get \eqref{abs},
which we wanted to prove. The inequality \eqref{ineq-lux} is obtained by homogeneity.  This concludes the proof of the proposition. \end{proof}

\medskip

It remains to adapt the remaining proof of Theorem \ref{mainresult1} to the present situation. We now take $\beta=\alpha$ in the definition of $S$, and $s=1$ when we use Lemma \ref{key-estimate-decomposition}. We choose $\eta$ so that $I-S$ has a small norm as an operator on the Banach space $\mathcal A^\Phi_\alpha (\mathbb B^n)$. We need to assume that $\int \Phi(|f|)d\nu_\alpha$ is bounded by $1$ to be able to use Proposition \ref{above-convex}. We leave the details to the reader. By homogeneity we obtain the same condition on norms as in the concave case:
\begin{equation*}
\|f\|_{\Phi, \alpha}\simeq  \inf_{\lbrace a_k\rbrace,\lbrace c_k\rbrace}\left\lbrace \|\lbrace f_{a_k,c_k}\rbrace\|_{l^{\Phi}_{\alpha, b}},\quad  \; f=\sum f_{a_k, c_k}\right\rbrace.
\end{equation*}

 This completes the proof of the theorem.

\section{Weak factorization theorems for Bergman-Orlicz spaces}
\label{section4}
In this section, we use atomic decomposition in order to obtain weak factorization theorems for functions in $\mathcal A^\Phi_\alpha (\mathbb B^n).$
We give two types of weak factorization for functions in $\mathcal L^\Phi_\alpha (\mathbb B^n).$ The first one is in terms of Bloch space and Bergman-Orlicz space and the second one is in terms of two Bergman-Orlicz spaces.

\subsection{Products of functions. The proof of Proposition \ref{incl-bloch}}
We suppose that $\Phi$ is a growth function of  lower type $p$. Since $\mathcal A^\Phi_\alpha (\mathbb B^n)\subset \mathcal A^p_\alpha (\mathbb B^n)$, we know that the product of $f\in \mathcal A^\Phi_\alpha (\mathbb B^n)$ and $g\in\mathcal B$ (Bloch space) is well defined as the product of a function in $\mathcal A^p_\alpha (\mathbb B^n)$ and a function in $\mathcal A^s_\alpha (\mathbb B^n)$ for all $1<s<\infty$ (since $\mathcal B\subset \mathcal A^s_\alpha (\mathbb B^n)$). So it is a function of $\mathcal A^q_\alpha (\mathbb B^n)$ for $q<p$. But this can be replaced by the sharp statement given by Proposition \ref{incl-bloch}. The following key lemma shows that the Bloch space is contained in the exponential class in the unit ball.

\begin{lemma}
\label{lem:exponential-class}
Let $\Phi(t)=\exp(t)-1$. There exists a constant $C$ such that
for any $f\in\mathcal B$,
$$\|f\|_{\Phi, \alpha}\leq C\|f\|_{\mathcal{B}}.$$
\end{lemma}
\begin{proof}
It is enough to show that there exists constants $\lambda>0$ and $C(\lambda)$ such that for any $f\in\mathcal B$, with $\|f\|_{\mathcal{B}}\neq 0,$

\begin{equation}
\label{eq:bloch-exponential}
    \int_{\mathbb B^n}\exp\left(\frac{|f(z)|}{\lambda \|f\|_{\mathcal{B}} }\right)d\nu_{\alpha}(z)\leq C (\lambda).
\end{equation}
We know that for $f\in\mathcal B$, we have
\begin{equation}
\label{eq:bloch-pointwise-estimate}
    |f(z)|\leq \log\left(\frac{4}{1-|z|^2}\right) \|f\|_{\mathcal{B}}, \qquad z\in \mathbb B^n.
\end{equation}
From (\ref{eq:bloch-pointwise-estimate}), we have
\begin{eqnarray*}
   \int_{\mathbb B^n}\exp\left(\frac{|f(z)|}{\lambda \|f\|_{\mathcal{B}} }\right)d\nu_{\alpha}(z) & \leq &\int_{\mathbb B^n}\exp\left( \frac 1\lambda\log\left(\frac{4}{1-|z|^2}\right)\right)d\nu_{\alpha}(z) \\
&\leq & C\int_0^1 \exp\left( \log\left(\frac{4}{1-r^2}\right)^{\frac 1\lambda}\right)(1-r^2)^\alpha 2rdr\\
&=&C4^{1/\lambda}\int_0^1(1-r)^{-\frac 1\lambda + \alpha}dr.
\end{eqnarray*}
We easily obtain (\ref{eq:bloch-exponential}) by taking $\lambda>\frac 1{1+\alpha}$. This finishes the proof of the lemma.\end{proof}

Since Lemma \ref{lem:exponential-class} shows that a function $f\in\mathcal B$ is in the exponential class, Proposition \ref{incl-bloch} follows from the use of  H\"older inequality for Orlicz spaces (see Proposition \ref{incl-prod} and \cite{VT}).

\subsection{Weak factorization with one factor in the Bloch space}
This subsection is devoted to the proof of Theorem \ref{mainresult3}.
 
\begin{proof}
 Let $\Phi\in \mathscr{L}_p$ and let  $\Psi(t)=\Phi\left(\frac{t}{\log(e+t)}\right)$. Since $t\mapsto \frac{t}{\log(e+t)} \in \mathscr{L}_p$, we know that $\Psi\in \mathscr{L}_p$.  Let $f\in \mathcal A^{\Psi}_\alpha(\mathbb
 B^n)$.  From Theorem \ref{mainresult1}, we know that there exist a sequence of points $\{a_k\}$ in $\mathbb B^n$ and a sequence of complex numbers   $\lbrace c_k\rbrace$ such that
\begin{equation}\label{eq:atomic-expression}
  f(z)=\sum_{k=1}^{+\infty}\frac{c_k}{(1-\langle z,a_k\rangle)^b},
  \qquad\quad z\in\mathbb B^n,
 \end{equation}
with
\begin{equation}
\label{eq:estimates-atomic-decomposition}
  \int_{\mathbb B^n} \Psi(|f(z)|)d\nu_\alpha(z)\simeq \sum_k(1-|a_k|^2)^{n+1+\alpha}\Psi
  \left(\frac{|c_k|}{(1-|a_k|^2)^{b}}\right).
 \end{equation}
 We assume that  $\|f\|_{\Psi,\alpha}\leq 1$ and prove \eqref{critical}.
Let
$$h:=h_k=\frac{c_k}{(1-\langle z,a_k\rangle)^b}. $$ 
We write $c$ and $a$ without index for simplification. We want to write each $h$ as a product $g\theta$, with
$$(1-|a|^2)^{n+1+\alpha}\Psi
  \left(\frac{|c|}{(1-|a|^2)^{b}}\right)\simeq \left(\int_{\mathbb B^n} \Phi(|g(z)|) d\nu_\alpha(z)\right)\|\theta\|_{\mathcal B}.$$ Indeed, if we find such a factorization for each term, the expression of $f$ as a sum of products that satisfies \eqref{critical} follows at once. The choice of factors will depend on the quantity $|c|(1-|a|^2)^{-b}$.
\begin{itemize}
\item
Assume that $|c|(1-|a|^2)^{-b}\leq 4$. Then
$$(1-|a|^2)^{n+1+\alpha}\Phi\left(\frac{|c|}{(1-|a|^2)^{b}}\right)\simeq (1-|a|^2)^{n+1+\alpha}\Psi\left(\frac{|c|}{(1-|a|^2)^{b}}\right).$$
We can take $g=h$ and $\theta=1$ since the left hand side is equivalent to $\int_{\mathbb B^n} \Phi(|g(z)|) d\nu_\alpha(z)$ by Lemma \ref{lem:norm-estimate1}.
\item
Assume that $|a|^2\leq 1-\eta$, for some $\eta\in (0, 1)$ which will be chosen below. We can take the same choice of factors since we still have  $|c|\leq C(1-|a|^2)^{b-\frac {n+1+\alpha}p}$.
\item
Assume that $|c|(1-|a|^2)^{-b}>4$ and $|a|^2>1-\eta$. Under the first condition, $\log (e+\frac{|c|}{(1-|a|^2)^b})\simeq \log (\frac{|c|}{(1-|a|^2)^b})$. We use the inequality
$$ \sum_k(1-|a_k|^2)^{n+1+\alpha}\Psi
  \left(\frac{|c_k|}{(1-|a_k|^2)^{b}}\right)\lesssim 1$$ and the fact that $\Psi$ is of lower type $p$ to remark that
  $$|c|\leq C(1-|a|^2)^{b-\frac{n+1+\alpha}{p}}$$
  for some uniform constant $C$. So, if we choose $\eta$ small enough, we have $|c|<1$. Let $$\delta:= \frac{|\log |c||}{b|\log (1-|a|^2)|}.$$
We have $0<\delta <1$ and
  $$(1-\delta)\log \left(\frac{4}{1-|a|^2}\right)\simeq \log \left(e+\frac{|c|}{(1-|a|^2)^b}\right).$$
  We choose $$\theta(z)= 1+(1-\delta)\log \left(\frac{4}{1-\langle a, z\rangle}\right).$$
  It is easy to see and classical that $\theta$ is uniformly in the Bloch class, with $\|\theta\|_{\mathcal B}\simeq 1$. So to conclude it is sufficient to prove the following lemma, which we use with $\lambda=\frac{c}{1-\delta}.$ We suppressed the constant $1$ before the logarithm for simplicity, which makes no harm for the bound above. The proof is identical for the bound below.
\end{itemize}
\begin{lemma}
\label{prop:factor-estimates} Let $\Phi\in \mathscr{L}_p$,  $a\in\mathbb B^n$ and $b>\frac{n+1+\alpha}p$. Then the function
$$g(z)= \frac{\lambda}{(1-\langle z, a\rangle)^b \log (\frac{4}{1-\langle a, z\rangle})}\qquad\quad (\lambda>0)$$
satisfies the inequalities
\begin{equation}\label{eq:estimate-factor}\int_{\mathbb B^n} \Phi(|g(z)|) d\nu_\alpha(z)\simeq (1-|a|^2)^{n+1+\alpha}\Phi\left(\frac{|\lambda|}{(1-|a|^2)^b \log (\frac{4}{1-|a|^2})}\right)\end{equation}
uniformly in $a$ and $\lambda$.
\end{lemma}
\begin{proof}
This is the analog of  Lemma \ref{lem:norm-estimate1}, but with an extra logarithmic factor. Recall that $|1-\langle z, a\rangle|\geq 1-|a|.$ It follows that for fixed $\eta>0$, with $\eta<1/8$, this factor is bounded below and above when $|a|\leq 1-\eta$. 
So it remains to consider the case when $|a|>1-\eta$.
For the lower bound we have a smaller quantity with the logarithm replaced by $\log(\frac{4}{1-|a|})$ which is equivalent to $\log(\frac{4}{1-|a|^2})$. We then use the lower estimate of Lemma \ref{lem:norm-estimate1} for the remaining function.

We now proceed to prove the upper bound in (\ref{eq:estimate-factor}). We mimic the proof of Lemma \ref{lem:norm-estimate1} but have now the supplementary factor
$$A:=\frac{\log\left(\frac{4}{1-|a|^2}\right)}{\log\left(\frac{4}{|1-\langle z, a\rangle|}\right)}.$$ It follows from elementary properties of the logarithm that
$$A\leq C_{\varepsilon}\left(\frac{1-|a|^2} {|1-\langle z, a\rangle|}\right)^{-\varepsilon}$$
for every $\varepsilon>0$. We choose $\varepsilon$ so that $b-\varepsilon>\frac{n+1+\alpha}p$.
From this point the proof is the same as for Lemma \ref{lem:norm-estimate1}, using the fact that the lower type property \eqref{eq:lowertype} is valid for $t\leq 2^b C_\varepsilon$, which is a bound when $t=C_\varepsilon(\frac{1-|a|^2} {|1-\langle z, a\rangle|})^{b-\varepsilon}$.

\end{proof}

 This proves \eqref{critical}. To  finish the proof of the theorem we need to prove \eqref{inequality-lux}. By homogeneity we may assume that $\|f\|_{\Psi,\alpha}=1$. By \eqref{critical}, it is sufficient to prove that $\|f_k\|_{\Phi, \alpha}\lesssim \int_{\mathbb B^n} \Phi(|f_k(z)|)d\nu_\alpha(z)$ which is a consequence of \eqref{upPhi}.

\end{proof}

\subsection{Application to the characterization of bounded small Hankel operators.}
As a corollary of Theorem \ref{mainresult3}, we obtain the following characterization of bounded Hankel operators from $\mathcal A^\Phi_\alpha (\mathbb B^n)$ into $\mathcal A^1_\alpha (\mathbb B^n)$. Recall that for $b\in \mathcal {A}_\alpha^2(\mathbb B^n)$, the small Hankel
 operator with symbol $b$ is defined for $f$ a bounded
 holomorphic function by $h_b(f):=P_\alpha(b\overline f)$. Here $P_\alpha$ is the orthogonal projection of the Hilbert space $L_\alpha^2( \mathbb B^n)$ onto its closed subspace $\mathcal {A}_\alpha^2(\mathbb B^n),$ called the
 Bergman projection, and it is given by
 \begin{equation} P_\alpha(f)(z)=\int_{\mathbb B^n}K_\alpha(z,\xi)f(\xi)d\nu_\alpha(\xi),\end{equation}
 where
$$K_\alpha(z,\xi)=\frac{1}{(1-\langle z,\xi \rangle)^{n+1+\alpha}}.$$
Let $\gamma>0$. We say that a growth function $\rho$ is of restricted upper type $\gamma$ on $[0,1]$ if there exists a constant $C$ such that
\begin{eqnarray}\label{rho-uppertype}
    \rho(st)\le Cs^\gamma\rho(t),
\end{eqnarray}
for $s>1$ and $st\leq 1$. We will call a weight, a growth function $\rho$  which is of restricted upper type $\gamma$, for some $\gamma>0$.

Now for $\alpha>-1$ and a weight $\rho$ (of restricted upper type $\gamma$), we define  the weighted Lipschitz space $\Gamma_{\alpha,\rho}(\mathbb B^n)$ as the space of holomorphic functions $f$ in $\mathbb B^n$ satisfying the following property: for some integer $k>\gamma (n+1+\alpha),$ there exists a positive constant $C>0$ such that
$$ |R^kf(z)|\le C(1-|z|^2)^{-k}\rho\left((1-|z|^2)^{n+1+\alpha}\right).$$
The Lipschitz space $\Gamma_{\alpha,\rho}(\mathbb B^n)$ is a Banach space under the following norm
$$\|f\|_{\Gamma_{\alpha,\rho}(\mathbb B^n)} = |f(0)|+ \sup_{z\in\mathbb B^n}\frac{|R^k f(z)|(1-|z|^2)^{k}}{\rho\left((1-|z|^2)^{n+1+\alpha}\right)}.$$
It was proved in \cite{ST2} that, as in the classical Lipschitz spaces,  these spaces are independent of $k$ and they are duals of Bergman-Orlicz spaces with concave Orlicz functions. More precisely, $\Gamma_{\alpha,\rho}(\mathbb B^n)$ can be identified as the dual of the Bergman-Orlicz space $\mathcal A^\Psi_\alpha (\mathbb B^n)$ with $\Psi^{-1}(t) = \frac t{\rho (\frac 1t)}.$ 

Using Theorem \ref{mainresult3}, we recover the following result proved in \cite{ST2} using a different approach.

\begin{corollary}
Let $\Phi\in \mathscr{L}_p$. A Hankel operator $h_b$ extends to a continuous operator from $\mathcal A^\Phi_\alpha (\mathbb B^n)$ to $\mathcal A^1_\alpha (\mathbb B^n)$ if and only if $b\in  \Gamma_{\alpha,\rho},$
where
$\rho (t)=\frac{1}{t\Psi^{-1}\left(\frac 1t\right)}$ with   $\Psi(t)=\Phi\left(\frac t{\log(e+t)}\right).$
\end{corollary}

\subsection{Weak factorizations with Bergman-Orlicz factors}
We give here a weak factorization theorem for functions in $\mathcal A^\Phi_\alpha (\mathbb B^n),$ with $\Phi\in \mathscr{L}_p$ in terms of products of functions in Bergman-Orlicz spaces.

\begin{theorem}
\label{mainresult4}
 Let $\Phi\in \mathscr{L}_p\cup \mathscr{U}^q$. Let $\Phi_1$ and $\Phi_2$ be two growth functions in either $\mathscr{L}_p$ or $\mathscr{U}^q$ such that
\begin{equation}\label{eq:volbergproduct}
\Phi^{-1}=\Phi_1^{-1}\times \Phi_2^{-1}.
\end{equation}
 Every function $f\in \mathcal A^{\Phi}_\alpha(\mathbb
 B^n)$ admits a decomposition
\begin{equation}
\label{eq:factorization2}
    f(z)=\sum_{k=1}^{+\infty}g_k(z)h_k(z),\quad z\in\mathbb B^n,
\end{equation}
 where each $g_k$ is in $\mathcal A^{\Phi_1}_\alpha(\mathbb
 B^n)$ and each  $h_k$ is in $\mathcal A^{\Phi_2}_\alpha (\mathbb B^n).$
Furthermore, if $\Phi\in \mathscr{L}_p$ then
\begin{equation}
\label{eq:weak-factorization-estimates}
\sum_{k=1}^{+\infty}\|g_k\|_{\Phi_1,\alpha}\|h_k\|_{\Phi_2,\alpha}\leq C\|f\|_{\Phi, \alpha},
\end{equation}
where $C$ is a positive constant independent of $f$.
\end{theorem}

\begin{proof}
Let $f\in \mathcal A^\Phi_\alpha (\mathbb B^n)$. We know from Theorem \ref{mainresult1} and Theorem \ref{mainresult2}, that there exists a sequence $\{a_k\}$ in $\mathbb B^n$ such that every $f\in  \mathcal A^\Phi_\alpha (\mathbb B^n)$ admits the following representation
$$f(z)=\sum_k\frac{c_k}{\left(1-\langle z,a_k\rangle\right)^{b}},$$
where $\{a_k\}$, $\{c_k\}$ belongs to the  space $l^{\Phi}_{\alpha, b}$ and the series converges in the norm topology of $\mathcal A^\Phi_\alpha (\mathbb B^n)$. We have
\begin{equation}
    \label{eq:estimate-final}
 \sum_{k=1}^{+\infty} (1-|a_k|^2)^{n+1+\alpha}\Phi\left(\frac{|c_k|}{(1-|a_k|^2)^{b}}\right)<\infty.
\end{equation}
Now take, for non zero $c_k$,
$$g_k(z)=\frac{(1-|a_k|^2)^{bs}}{\left(1-\langle z,a_k\rangle\right)^{bs}}\Phi^{-1}_1\left(\Phi\left(\frac{|c_k|}{(1-|a_k|^2)^{b}}\right)\right)e^{iArg (c_k)}$$
and
$$h_k(z)=\frac{(1-|a_k|^2)^{bt}}{\left(1-\langle z,a_k\rangle\right)^{bt}}\Phi^{-1}_2\left(\Phi\left(\frac{|c_k|}{(1-|a_k|^2)^{b}}\right)\right)$$
where $s,t>0$ with $s+t=1$. It is clear, using (\ref{eq:volbergproduct}), that (\ref{eq:factorization2}) holds. Using Lemma \ref{lem:norm-estimate1} or Remark \ref{rem:norme-estimate2}, we easily see that $g_k\in \mathcal A^{\Phi_1}_\alpha (\mathbb B^n)$ and $h_k\in \mathcal A^{\Phi_2}_\alpha (\mathbb B^n)$.

 It remains to prove (\ref{eq:weak-factorization-estimates}).  Let $\Phi\in \mathscr{L}_p$. By homogeneity, we may suppose $\|f\|_{\Phi, \alpha}=1$. We then have to show that there exists a constant $C$, independent of $f$, so that
$$\sum_k \|g_k\|_{\Phi_1,\alpha} \|h_k\|_{\Phi_2,\alpha}\leq C.$$
Using Lemma \ref{lem:bounded-concave} or Lemma \ref{lem:bounded-convex}  and   Theorem \ref{thm:estimate-classic}  again,   with $b$ large enough, we have, for $\Phi_1,\Phi_2\in \mathscr{L}_p\cup\mathscr{U}^q$
\begin{eqnarray}\label{eq:estimate51}
    \|g_k\|_{\Phi_1,\alpha}\leq C \frac{\Phi^{-1}_1\left(\Phi\left(\frac{|c_k|}{(1-|a_k|^2)^{b}}\right)\right)}{\Phi^{-1}_1\left(\frac 1{(1-|a_k|^2)^{n+1+\alpha}}\right)}
\end{eqnarray}
and
\begin{eqnarray}\label{eq:estimate2}
     \|h_k\|_{\Phi_2,\alpha}\leq C \frac{\Phi^{-1}_2\left(\Phi\left(\frac{|c_k|}{(1-|a_k|^2)^{b}}\right)\right)}{\Phi^{-1}_2\left(\frac 1{(1-|a_k|^2)^{n+1+\alpha}}\right)}.
\end{eqnarray}
Now, using (\ref{eq:estimate51}), (\ref{eq:estimate2}) and (\ref{eq:volbergproduct}), we have
\begin{multline*}   \sum_{k}\|g_k\|_{\Phi_1,\alpha}\|h_k\|_{\alpha,\Phi_2}  \le \\ C\sum_{k} \frac{|c_k|}{(1-|a_k|^2)^{b}}\frac 1{\Phi^{-1}\left(\frac 1{(1-|a_k|^2)^{n+1+\alpha}}\right)}\\
= C\sum_{k} \frac{\Phi^{-1}\left(\frac {d_k}{(1-|a_k|^2)^{n+1+\alpha}}\right)} {\Phi^{-1}\left(\frac 1{(1-|a_k|^2)^{n+1+\alpha}}\right)},
\end{multline*}
where $d_k=(1-|a_k|^2)^{n+1+\alpha}\Phi\left(\frac{|c_k|}{(1-|a_k|^2)^{b}}\right)$. Since $\|f\|_{\Phi, \alpha}=1$, there exists a uniform constant $C$ such that $\int_{\mathbb B^n}\Phi(|f(z)|)d\nu_\alpha(z)\leq C$ (see (\ref{upPhi})). By (\ref{eq:estimate-final}) the series $\{d_k\}$ converges in $l^1$. This implies, without loss of generality that we may assume $\{d_k\}$ is bounded by $1$. Since $u\mapsto \frac {\Phi(u)}u$ is non-increasing, we have that $u\mapsto \frac {\Phi^{-1}(u)}{u}$ is non-decreasing, hence
$$\Phi^{-1}(uv)\leq u\Phi^{-1}(v),\qquad 0\leq u\leq 1,\: v\geq 0.$$
From this, we have
\begin{multline*}   \sum_{k}\|g_k\|_{\Phi_1,\alpha}\|h_k\|_{\Phi_2,\alpha}  \le \\ C \sum_{k=1}^{+\infty} d_k=C\sum_{k=1}^{+\infty}(1-|a_k|^2)^{n+1+\alpha}\Phi\left(\frac{|c_k|}{(1-|a_k|^2)^{b}}\right)\leq C.
\end{multline*}
This finishes the proof.
\end{proof}
\begin{remark}
Weak factorization theorems with Bergman-Orlicz factors in the case where $\Phi\in\mathscr{U}$ are considered in an upcoming paper by the third author and R. Zhao.
\end{remark}

\subsection*{Acknowledgements}
E. Tchoundja would like to thank the support of  the Fulbright Scholar Program which has supported his visit to Washington University in St. Louis
where this work has been done.


\begin{thebibliography}{HD}
\normalsize
\baselineskip=17pt


\bibitem{BG}
A. Bonami and S. Grellier, \emph{Hankel operators
and weak factorization for Hardy-Orlicz spaces}, Colloquium Math., 
{118}, no. 1 (2010), 107 -- 132.

\bibitem{BS}
A. Bonami and B. Sehba, B., \emph{Hankel operators between Hardy-Orlicz spaces and products of holomorphic functions}, Rev. Math. Arg. Vol. {50}, no. 2  (2009), 187--199.

\bibitem{DHZZ}
Y. Deng, L. Huang, T. Zhao and D. Zheng, \emph{Bergman projection and Bergman spaces}, J. Oper. Theor. {46} (2001), 3--24.

\bibitem{LLQR}
P. Lef\`evre, D. Li, H. Queff\'ellec and L. Rodriguez-Piazza, \emph{Composition operators on Hardy-Orlicz spaces}, Memoirs of the AMS, {207},  (2010) no. 974.


\bibitem{PauZhao2}
J. Pau, R. Zhao, \emph{Weak factorization and Hankel forms
for weighted Bergman spaces on the unit ball}, Math. Ann. {363}, (2015) no. 1-2, 363--383.

\bibitem{RR}
M. M. Rao and Z. D. Ren, \emph{Theory of Orlicz functions},
Pure and Applied Mathematics {146}, Marcel Dekker, Inc. (1991).



\bibitem{sehbastevic}
B. F. Sehba and S. Stevic, \emph{On some product-type operators from
Hardy-Orlicz and Bergman-Orlicz spaces to weighted-type spaces},
Appl. Math. Comput. {233}, 565--581 (2014).


\bibitem{ST2}
B. Sehba and E. Tchoundja, \emph{Duality for large Bergman-Orlicz spaces and Hankel operators between holomorphic Bergman-Orlicz spaces on the unit ball},  Complex Var. Elliptic Equ. {62} (2017), no. 11, 1619--1644.


\bibitem{VT}
A. L. Volberg and V. A. Tolokonnikov, \emph{Hankel operators and problems of best approximation of unbounded functions.} Translated from Zapiski Nauchnykh Seminarov Leningradskogo Matematicheskogo Instituta im. V.A. Steklova AN SSSR,  {141} (1985), 5--17.


\bibitem{KZ} K. Zhu,
\emph{Spaces of holomorphic functions in the unit ball,} Graduate Texts in Mathematics 226, Springer  Verlag (2004).

\end{thebibliography}
\end{document}